\documentclass{amsart}

\usepackage{latexsym,enumerate}
\usepackage{amsmath,amsthm,amsopn,amstext,amscd,amsfonts,amssymb}
\usepackage[ansinew]{inputenc}
\usepackage{verbatim}
\usepackage{graphicx}
\usepackage{pstricks}
\usepackage{wasysym}
\usepackage[all]{xy}
\usepackage{epsfig} 
\usepackage{graphicx,psfrag} 
\usepackage{subfigure}
\usepackage{pstricks,pst-node}

\usepackage{multicol}
\usepackage{color}
\usepackage{colortbl}

\newcommand{\CC}{\mathbb{C}}
\newcommand{\NN}{\mathbb{N}}

\newcommand{\RR}{\mathbb{R}}
\newcommand{\ZZ}{\mathbb{Z}}

\newcommand{\DD}{\mathbb{D}}

\newcommand{\HH}{\mathbb{H}}

\newtheorem{theorem}{Theorem}[section]
\newtheorem{lemma}[theorem]{Lemma}
\newtheorem{proposition}[theorem]{Proposition}

\newtheorem{definition}[theorem]{Definition}
\newtheorem{example}[theorem]{Example}
\newtheorem{remark}[theorem]{Remark}

\DeclareMathOperator{\diam}{diam}
\DeclareMathOperator{\Vol}{Vol}
\DeclareMathOperator{\inj}{inj}

\DeclareMathOperator{\Arcsinh}{Arcsinh}

\DeclareMathOperator{\Arccosh}{Arccosh}

\newcommand{\spb}[1]{\smallskip}
\newcommand{\mpb}[1]{\medskip}
\newcommand{\bpb}[1]{\bigskip}

\newcommand{\p}{\partial}

\renewcommand{\a}{\alpha}
\renewcommand{\b}{\beta}
\newcommand{\e}{\varepsilon}
\renewcommand{\d}{\delta}
\newcommand{\D}{\Delta}
\newcommand{\g}{\gamma}
\newcommand{\G}{\Gamma}

\renewcommand{\l}{\lambda}
\renewcommand{\L}{\Lambda}
\renewcommand{\O}{\Omega}
\newcommand{\s}{\sigma}

\begin{document}
\DeclareGraphicsExtensions{.jpg,.pdf,.mps,.png}

\title{Isoperimetric inequalities in Riemann surfaces and graphs}

\author[\'{A}lvaro Mart\'{\i}nez-P\'erez]{\'{A}lvaro Mart\'{\i}nez-P\'erez$^{(1)}$}
\address{ Facultad CC. Sociales de Talavera,
Avda. Real Fábrica de Seda, s/n. 45600 Talavera de la Reina, Toledo, Spain}
\email{alvaro.martinezperez@uclm.es}
\thanks{$^{(1)}$ Supported in part by a grant
from Ministerio de Econom{\'\i}a y Competitividad (MTM2015-63612P), Spain.
}

\author[Jos\'e M. Rodr{\'\i}guez]{Jos\'e M. Rodr{\'\i}guez$^{(2)}$}
\address{Departamento de Matem\'aticas, Universidad Carlos III de Madrid,
Avenida de la Universidad 30, 28911 Legan\'es, Madrid, Spain}
\email{jomaro@math.uc3m.es}
\thanks{$^{(2)}$ Supported in part by two grants
from Ministerio de Econom{\'\i}a y Competititvidad, Agencia Estatal de
Investigación (AEI) and Fondo Europeo de Desarrollo Regional (FEDER) (MTM2016-78227-C2-1-P and MTM2017-90584-REDT), Spain.}

\date{\today}


\begin{abstract}
A celebrated theorem of Kanai states that quasi-isometries preserve isoperimetric inequalities between uniform Riemannian manifolds (with positive injectivity radius) and graphs.
Our main result states that we can study the (Cheeger) isoperimetric inequality in a Riemann surface by using a graph related to it, even if the surface has injectivity radius zero
(this graph is inspired in Kanai's graph, but it is different form it).
We also present an application relating Gromov boundary and isoperimetric inequality.
\end{abstract}

\maketitle{}



\section{Introduction}

Isoperimetric inequalities are of interest in pure and applied mathematics (see, e.g., \cite{C2}, \cite{Po}).
There are close connections between isoperimetric inequality and some conformal invariants of Riemannian manifolds and graphs, namely
Poincar\'e-Sobolev inequalities,
the bottom of the spectrum of the Laplace-Beltrami operator, the exponent of convergence,
and the Hausdorff dimensions of the sets of both bounded geodesics and escaping geodesics in a negatively curved surface
(see \cite{BJ}, \cite[p.228]{Bu}, \cite{Ch}, \cite{FM1}, \cite{FM2}, \cite{FMP1}, \cite{FMP2}, \cite{FR1}, \cite{MRT}, \cite{P}, \cite[p.333]{S}).
The Cheeger isoperimetric inequality is closely related to the project of Ancona
on the space of positive harmonic functions of Gromov-hyperbolic manifolds and graphs (\cite{A1}, \cite{A2} and \cite{A3}).
In fact, in the study of the Laplace operator on a hyperbolic manifold or graph $X$,
Ancona obtained in these three last papers interesting results, under the additional assumption that the bottom of the spectrum of the Laplace spectrum $\l_1(X)$ is positive.
The well-known Cheeger inequality $\l_1(X) \ge \frac14 \,h(X)^2$, where $h(X)$ is the isoperimetric constant of $X$,
guarantees that $\l_1(X)>0$ when $h(X) > 0$ (see \cite{Bu1} for a converse inequality).
Hence, the results of this paper are useful in order to obtain these Ancona's results.

Given any Riemannian $n$-manifold $M$, the Cheeger isoperimetric constant of $M$ is defined as
\[h(M) = \inf_\Omega \frac{\Vol_{n-1,M}(\partial \Omega)}{\Vol_{n,M}(\Omega)} \,,\]
where $\Omega$ ranges over all non-empty bounded open subsets of $M$,
and $\Vol_{k,M}(B)$ denotes the $k$-dimensional Riemannian volume in $M$ of the set $B$.
We write $\Vol_{2,M}=A_M$ and $\Vol_{1,M}=L_M$.

Given any graph $\G=(V,E)=(V(\G),E(\G))$, let us consider the natural length metric $d_{\,\Gamma}$ where every edge has length 1.
For any graph $\Gamma$, any vertex $v\in V$ and any $k\in \mathbb{N}$, let 
$B_k(v)$ and $\bar{B}_k(v)$ denote the open and closed balls, respectively.

The combinatorial Cheeger isoperimetric constant of $\Gamma$ is defined to be
\[h(\Gamma) = \inf_A \frac{|\partial A|}{|A|} \,,\]
where $A$ ranges over all non-empty finite subsets of vertices in $\Gamma$,
$\partial A = \{v \in \Gamma \, | \, d_{\,\Gamma}(v,A) = 1\}$ and $|A|$ denotes the cardinality of $A$.

We say that a Riemannian manifold or graph $X$ satisfies the (Cheeger or linear) \emph{isoperimetric inequality} (LII) if $h(X)>0$, since in this case
\begin{equation} \label{eq:1.1}
\Vol_{n,M}(A) \le h(X)^{-1} \Vol_{n-1,M}(\partial A) ,
\end{equation}
for every bounded open set $A \subseteq X$ if $X$ is a Riemannian $n$-manifold, and
$$
|A| \le h(X)^{-1} |\partial A| ,
$$
for every finite set $A \subseteq V(X)$ if $X$ is a graph.

Along the paper, we just consider manifolds and graphs $X$ which are connected.
This is not a loss of generality, since if $X$ has connected components $\{X_j\}$, then $h(X)=\inf_j h(X_j)$.

Let $(X,d_X)$ and $(Y,d_Y)$  be two metric spaces. A map $f: X\longrightarrow Y$ is said to be
an $(\alpha, \beta)$-\emph{quasi-isometric embedding}, with constants $\alpha\geq 1,\
\beta\geq 0$, if for every $x, y\in X$:
$$
\alpha^{-1}d_X(x,y)-\beta\leq d_{\,Y}(f(x),f(y))\leq \alpha \, d_X(x,y)+\beta.
$$
The function $f$ is $\varepsilon$-\emph{full} if
for each $y \in Y$ there exists $x\in X$ with $d_Y(f(x),y)\leq \varepsilon$.

A map $f: X\longrightarrow Y$ is said to be
a \emph{quasi-isometry}, if there exist constants $\alpha\geq 1,\
\beta,\varepsilon \geq 0$ such that $f$ is an $\varepsilon$-full
$(\alpha, \beta)$-quasi-isometric embedding.
Two metric spaces $X$ and $Y$ are \emph{quasi-isometric} if there exists
a quasi-isometry $f:X\longrightarrow Y$.
One can check that to be quasi-isometric is an equivalence relation.

A graph $\Gamma$ is said to be $\mu$-\emph{uniform} if each vertex $p$ of $V$ has at most $\mu$ neighbors, i.e.,
\[\sup\big\{|N(p)| : \, p\in V(\G)\big\}\leq \mu.  \]
If a graph $\Gamma$ is $\mu$-uniform for some constant $\mu$ we say that $\Gamma$ is \emph{uniform}.

A complete Riemannian $n$-manifold $M$ is \emph{uniform} if it has positive injectivity radius and a lower bound on its Ricci curvature.

The \emph{injectivity radius} inj$(x)$ of \emph{$x\in M$} is defined as the supremum of those $r>0$ such that $B_r(x)$ is simply connected or, equivalently,
as half the infimum of the lengths of the (homotopically non-trivial) loops based at $x$.
The \emph{injectivity radius} inj$(M)$ \emph{of $M$} is the infimum over $x\in M$ of inj$(x)$.

A celebrated theorem of Kanai in \cite{K} states that quasi-isometries preserve isoperimetric inequalities
between uniform Riemannian manifolds and graphs.
A main ingredient in Kanai's proof is the following interesting fact:
we can study the LII in a Riemannian manifold $M$ by studying the LII in a graph related to it, if $M$ is uniform.
Kanai's result also holds with different hypotheses in the context of non-exceptional Riemann surfaces
(although it is not possible to relate the LII in the Riemann surface with the LII in the graph provided by Kanai)
\cite[Theorem 1.2]{CGPR}, \cite[Theorem 5.1]{GPPRT}.

A non-exceptional Riemann surface $S$ is a Riemann surface whose universal covering space is
the unit disk $\DD=\{z\in\CC:\; |z|<1\}$, endowed with its Poincar\'e
metric (also called the hyperbolic metric), i.e., the
metric obtained by projecting the Poincar\'e metric of the unit disk
$$
ds = \frac {2\,|dz|}{1-|z|^2} \,.
$$
With this metric, $S$ is a complete Riemannian manifold with constant curvature $-1$.
The only Riemann surfaces which are left out (the exceptional Riemann surfaces) are the sphere, the plane, the punctured plane and the tori.

A natural context to apply Kanai's results are Riemann surfaces endowed with their Poincar\'e metrics, since they have constant negative curvature.
However, these surfaces usually have isolated singularities which are cusps, and thus, injectivity radius equal to zero. That means that, unfortunately, it is not possible to apply Kanai's results in this case.
Recall that the Poincar\'e metric plays a main role in geometric function theory, since if $R,S$ are Riemann surfaces with their Poincar\'e metrics, then $d_S(f(w),f(z)) \le d_R(w,z)$ for every holomorphic map $f:R \rightarrow S$ and every $z,w \in R$.

Our main result is Theorem \ref{th:graph}, that states that we can study the LII in a Riemann surface by using a graph related to it, even if the surface has injectivity radius zero
(this graph is inspired in Kanai's graph, but it is different form it).
Thus, Theorem \ref{th:graph} shows that the main tool in Kanai's proof also works for all non-exceptional Riemann surfaces,
even if the injectivity radius is zero (although, unfortunately, the LII is not preserved by quasi-isometries between Riemann surfaces with injectivity radius zero, see \cite[Example 2.3]{CGPR}).
Although this result is interesting by itself,
we present in the last section an application of Theorem \ref{th:graph} relating Gromov boundary and LII (see Theorem \ref{th:sufic}).

Finally, we want to remark that \cite[Theorem 5.5]{GPPR} gives that every orientable and complete Riemannian surface with pinched negative curvature
(with Gaussian curvature $K$ satisfying $-k_2^2 \le K \le -k_1^2<0$)
is bilipschitz equivalent to a non-exceptional Riemann surface (and therefore with constant negative curvature $-1$).
\cite[Theorem 5.5]{GPPR} shows that it suffices to work with surfaces of curvature $-1$ (instead of pinched negative curvature) in order to check LII, since this inequality is invariant by bilipschitz maps.
This fact enlarges the scope of Theorems \ref{th:graph} and \ref{th:sufic}.

\section{Background and technical results}

A {\it geodesic domain} in a non-exceptional Riemann surface $S$ is a domain $G\subset S$ (which is not simply or doubly connected)
such that $\p G$ consists of finitely many simple closed geodesics, and $A_S(G)$ is finite.
$G$ does not have to be relatively compact since it may ``surround" finitely many cusps.
We can think of a cusp as a boundary geodesic of zero length.
Recall that if $\g$ is a closed curve in $S$ and $[\g]$ denotes its free homotopy class in $S$, then there is a unique simple closed geodesic of minimal length in the
class, unless $\g$ is homotopic to zero or surrounds only a cusp;
in this case it is not possible to find such geodesic because there are curves in the class with arbitrarily small length.

In \cite[Lemma 1.2]{FR1} it was proved the following.

\begin{lemma} \label{l:FR1}
If a non-exceptional Riemann surface $S$ satisfies
$$
A_S (G) \le c \, L_S (\p G)
$$
for some constant $c$ and every geodesic domain $G$ in $S$, then $S$ has LII.
\end{lemma}

In fact, if we define $h_g(S)$ as
$$
h_g(S) = \inf_{G}  \frac{L_S (\p G)}{A_S (G)}\,,
$$
where $G$ ranges over all geodesic domains in $S$, then \cite[Theorem 7]{Matsuzaki} gives
$$
h(S)^{-1}\le h_g(S)^{-1}+1 .
$$

\begin{definition}
Given a non-exceptional Riemann surface $S$, a subset $\O\subset S$ and a positive constant $\d$, let us denote by $\p_\d \O$ the set of points $z\in \p\O$ such that the connected component of $\p\O$ containing $z$ has length at least $\d$.
A non-exceptional Riemann surface $X$ is $\d$-\emph{regular}
if there exists a positive constant $c$ such that every geodesic domain $G$ in $X$ satisfies that $L(\p_\d G) \ge c\, CC(\p G\setminus \p_\d G)$
(where $CC(A)$ denotes the cardinality of the set of connected components of $A$).
We say that $X$ is \emph{regular} if it is $\d$-regular for some $\d > 0$.
\end{definition}

\begin{theorem} \label{t:regular}
If a non-exceptional Riemann surface is not regular, then it does not have LII.
\end{theorem}

\begin{proof}
Since the non-exceptional Riemann surface $S$ is not regular, for each $n\in\NN$ there exists a geodesic domain $G_n$ in $S$ with
$$
L_S (\p_{1/n} G_n) < \frac1n \, CC(\p G_n\setminus \p_{1/n} G_n) .
$$
Hence, since $L_S(\p G_n\setminus \p_{1/n} G_n)< \frac1n \, CC(\p G_n\setminus \p_{1/n} G_n)$, we have
\begin{equation} \label{eq:reg1}
\begin{aligned}
L_S(\p G_n)=L_S(\p_{1/n} G_n) + L_S(\p G_n\setminus \p_{1/n} G_n)< \frac2n \, CC(\p G_n\setminus \p_{1/n} G_n).
\end{aligned}
\end{equation}

If $m$, $p$ and $g$ denote the cardinality of the simple closed geodesics in $\p G_n$, the cusps surrounded by $G_n$ and the genus of $G_n$, respectively, it is well-known that Gauss-Bonnet Theorem gives $A_S(G_n)= 2\pi(m+p-2+2g)$ and so, $m+p-2+2g \ge 1$.
Note that $m = CC(\p G_n)$.
We are going to prove
\begin{equation} \label{eq:reg2}
m \le 3(m+p-2+2g).
\end{equation}
If $g \ge 1$, then $m+p-2+2g \ge m+p \ge m$.
Assume that $g=0$.
If $p \ge 2$, then $m+p-2+2g = m+p-2 \ge m$.
If $p=1$, then, since $m-1 = m+p-2+2g \ge 1$, we have $m \ge 2$ and $m \le 2(m-1) = 2(m+p-2+2g)$.
If $p=0$, then, since $m-2 = m+p-2+2g \ge 1$, we have $m \ge 3$ and $m \le 3(m-2) = 3(m+p-2+2g)$.
Therefore, \eqref{eq:reg2} holds.

Now, \eqref{eq:reg1} and \eqref{eq:reg2} give
$$
\begin{aligned}
\frac{n}{2} \, L_S (\p G_n)
& < CC(\p G_n\setminus \p_{1/n} G_n)
\le m
=  \frac{m}{2\pi(m+p-2+2g)}\, A_S(G_n)
\le \frac{3}{2\pi }\, A_S(G_n),
\\
A_S(G_n)
& > \frac{n\pi}{3} \, L_S (\p G_n) .
\end{aligned}
$$
Thus, $S$ does not have LII.
\end{proof}

\begin{remark} \label{r:regular}
Theorem \ref{t:regular} shows that, in order to study LII in non-exceptional Riemann surfaces, it suffices to consider regular surfaces.
\end{remark}

Next, we need to consider also bordered Riemann surfaces whose boundary
is a finite union of disjoint simple closed curves.
They will always arise as closed subsets of some open non-exceptional Riemann surface, and \textit{we
give them the metric induced from the Poincar\'e metric of the host open surface.}
Such induced metric has of course curvature $K = -1$.
An example of such a bordered Riemann surface is the closure of any geodesic domain.

\smallskip

A {\it Y-piece} is a compact bordered Riemann
surface which is topologically a sphere without three open disks and whose boundary curves are simple closed geodesics.
Given three positive numbers $a,\,b,\,c$, there is a unique (up to
conformal mapping) Y-piece such that their boundary curves have
lengths $a,\,b,\,c$ (see, e.g., \cite[p.410]{R}).
They are a standard tool for constructing Riemann
surfaces. A clear description of these Y-pieces and their use is given
in \cite[Chapter X.3]{C1} and \cite[Chapter 1]{Bu}.

\smallskip

A {\it generalized Y-piece} is a bordered or non-bordered
Riemann surface which is topologically a sphere without $n$ open
disks and $m$ points, with integers $n,m\ge 0$ such that $n+m=3$,
so that the $n$ boundary curves are
simple closed geodesics and the $m$ deleted points are cusps.
Observe that a generalized Y-piece is topologically
the union of a Y-piece and $m$ cylinders, with $0\le m\le 3$.
It is clear that the interior of every generalized Y-piece
is a geodesic domain. Furthermore,
it is  known that the closure of every geodesic domain is a finite
union (with pairwise disjoint interiors) of generalized
Y-pieces \cite[Proposition 3.2]{AR}.

\smallskip

The following example shows that regularity is not a sufficient condition in order to have LII.

\begin{example} Let $X$ be a non-exceptional Riemann surface built as follows:
Consider the generalized $Y$-pieces $\{Y_i\}_{i\in \ZZ}$ with one cusp and such that $\partial Y_i=\eta^i_1\cup\eta^i_2$ with $L(\eta^i_1)=L(\eta^i_2)=1$.
Let $X$ be the surface obtained
by identifying $\eta^i_2$ with $\eta^{i+1}_1$ for every $i\in \ZZ$.
It is clear that if $G_n=\cup_{i=1}^nY_i$, then $A(G_n)=2\pi n$, $L(\partial G_n)=2$ and $X$ does not have LII.
However, for any geodesic domain $G$ and any $0< \d <1$, $L(\partial_\d G)= L(\partial G) \ge 0 = CC(\p G\setminus \p_{\d} G)$, and $X$ is regular.
\end{example}

\bpb

%
%
%
%
%
%
%
%
%
%
%
%

A \emph{collar} in a non-exceptional Riemann surface $S$
about a simple closed geodesic $\g$ is a doubly connected
domain in $S$ ``bounded" by two
Jordan curves (called the boundary curves of the collar) orthogonal
to the pencil of geodesics emanating from $\g$;
such collar is equal to $\{p\in S:\,d_S(p,\g) < d\}$,
for some positive constant $d$.
The constant $d$ is called the \emph{width} of the collar.

Let $S$ be a non-exceptional Riemann surface with a cusp $r$
(if $S\subset\CC$, every isolated point in $\p S$ is a cusp).
A {\it collar} in $S$ about $r$ is a doubly connected domain in $S$
``bounded" both by $r$
and a Jordan curve (called the boundary curve of the collar) orthogonal
to the pencil of geodesics emanating from $r$.
It is well-known that the
length of the boundary curve is equal to the area of the collar
(see, e.g., \cite{Be}).
A collar of area $\l$ about $r$ is called a $\l$-collar.
We denote by $C(r,\l)$ the $\l$-collar of the cusp $r$ with area $0< \l \le 2$.

We will use several times the following result known as Collar Lemma (see \cite{R}).

\begin{lemma} \label{l:CollarLemma}
If $\g$ is a simple closed geodesic in a non-exceptional Riemann surface $S$, then
there exists a collar about
$\g$ of width $w$, where
$\cosh w = \coth (L_S(\g)/2)$.
\end{lemma}

\begin{remark} \label{r:CollarLemma}
Along this paper,
$\g$ will denote a simple closed geodesic in $S$
and $w$ the width of the collar of $\g$, where
$\cosh w = \coth (L_S(\g)/2)$.

Denote by $C(\g,h)$ the collar of $\g$ of width $h$ and by $C(\g)$ the collar
of $\g$ of width $w$.
It is well-known that if $\g_1$ and $\g_2$ are disjoint simple closed geodesics,
then $C(\g_1) \cap C(\g_2)=\emptyset$.

For each cusp there exists a $2$-collar and $2$-collars of different cusps are disjoint.
Besides, the collar $C(\g)$ of the simple closed geodesic $\g$ does not intersect
the $2$-collar of a cusp
(see  \cite{R}, \cite{Sh} and \cite[Chapter 4]{Bu}).
\end{remark}

We will use the thick-thin decomposition of Riemann surfaces given by Margulis Lemma (see, e.g., \cite[p.107]{BGS}).
Concretely,
for any $\e< \Arcsinh 1$ any Riemann surface, $S$, can be partitioned into a thick part,
$S_{\e}:=\{z\in S : \inj(z) \ge \varepsilon \}$, and a thin part, $S\setminus S_{\e}$,
whose connected components are either collars of cusps or collars of simple closed geodesics of length less than $2\e$.
In fact, \cite[Lemma 4.9]{CGPR} gives the following.

\begin{lemma} \label{l:ts}
Let $S$ be a non-exceptional Riemann surface and $z\in S$.
If $\inj (z) < \Arcsinh 1$, then the shortest geodesic loop
$\eta$ with base point $z$  is contained either in the $2$-collar of a cusp
or in the collar $C(\g)$ of a simple closed geodesic $\g$.
\end{lemma}

We collect below a well-known hyperbolic trigonometric formula (see, e.g., \cite[p.454]{Bu})
which will be useful.

\begin{proposition} \label{l:Bu}
The following formula holds for polygons on the unit disk (and then for simply connected polygons on any non-exceptional Riemann surface).
%
%
%

Let us consider a geodesic quadrilateral with three right angles and let $\phi$ the other angle.
If $\a,\b$ are the lengths of the sides which meet with angle $\phi$
and $a$ is the length of the opposite side to the side with length $\a$, then
$\sinh \a= \sinh a \cosh \b$.
\end{proposition}

\begin{lemma} \label{l:0}
Let $S$ be a non-exceptional Riemann surface, $0<\e < \Arcsinh 1$ and $C$ a connected component of $\{z\in S\,| \; \inj(z)< \e \}$.
By Margulis Lemma, $C$ is a collar.

$(1)$ If $C$ is a collar of a simple closed geodesic $\g$, then $C=C(\g,h) \subset C(\g)$ and $h=\Arccosh \frac{\sinh \e}{\sinh (L_S(\g)/2)}$.

$(2)$ If $C$ is a collar of a cusp $r$, then $C=C(r,\l) \subset C(r,2)$ and $\l=2 \sinh \e$.
\end{lemma}

\begin{proof}
Assume first that $C=C(\g,h)$ is a collar of a simple closed geodesic $\g$.
Fix $z \in \p C$.
Since $\inj (z) = \e < \Arcsinh 1$, Lemma \ref{l:ts} gives that the shortest geodesic loop
$\s$ with base point $z$ is contained in $C(\g)$.
Let $l:=L_S(\g)$.
By Proposition \ref{l:Bu}, we have
$$
\sinh \e
= \sinh \inj(z)
= \sinh (l/2) \cosh h,
\qquad
h
= \Arccosh \frac{\sinh \e}{\sinh (l/2)} \,.
$$
Collar Lemma gives that there exists a collar about $\g$ of width $w$, where $\cosh w = \coth (l/2)$.
Thus,
$$
\cosh h
= \frac{\sinh \e}{\sinh (l/2)}
< \frac{1}{\sinh (l/2)}
< \frac{\cosh (l/2)}{\sinh (l/2)}
= \cosh w ,
$$
$h<w$ and $C$ is contained in $C(\g)$.

Assume now that $C=C(r,\l)$ is a collar of a cusp $r$.
Fix $z \in \p C$.
Since $\inj (z) = \e < \Arcsinh 1$, Lemma \ref{l:ts} gives that the shortest geodesic loop
$\s$ with base point $z$ is contained in $C(r,2)$.

As usual, consider a fundamental domain for $S$ in the upper half-plane $\HH$
contained in $\{\zeta\in\HH: \, 0\le \Re \zeta \le 1\}$
and such that $\{\zeta\in\HH: \, 0\le \Re \zeta \le 1,\,\Im \zeta > 1/2\}$
corresponds to $C(r,2)$.
Thus, $\{\zeta\in\HH: \, 0\le \Re \zeta \le 1,\,\Im \zeta > 1/\l\}$
corresponds to $C(r,\l)$.
Without loss of generality we can assume that $i/\l$ corresponds to $z$.
Thus, we can represent $\s$ in the upper half-plane by means of a
geodesic with endpoints $i/\l$ and $1+i/\l$.
We have
$$
\sinh \e
= \sinh \inj(z)
= \sinh \frac{L_S(\s)}2
= \sinh \frac{d_\HH(i/\l,1+i/\l)}2
= \frac{\l}{2}\,,
$$
$\l = 2 \sinh \e<2$ and $C$ is contained in $C(r,2)$.
\end{proof}

\begin{lemma} \label{l:dc}
Let $S$ be a non-exceptional Riemann surface, $0<\e < \Arcsinh 1$ and $C$ a connected component of $\{z\in S\,| \; \inj(z)< \e \}$.
If $\eta$ is a connected component of $\p C$, then $L_S(\eta) \le 2 \sinh \e < 2$.
\end{lemma}

\begin{proof}
By Margulis Lemma, $C$ is a collar.

Assume first that $C$ is a collar of a simple closed geodesic $\g$.
Lemma \ref{l:0} gives that $C=C(\g,h) \subset C(\g)$ and $h=\Arccosh \frac{\sinh \e}{\sinh (l/2)}$ with $l=L_S(\g)$.
We have
$$
L_S(\eta)
= l \cosh h
= l \, \frac{\sinh \e}{\sinh (l/2)} \,.
$$
Since the function $f(t)=t-\tanh t$ satisfies $f'(t)=1-(\cosh t)^{-2} > 0$ for every $t > 0$, we conclude $f(t) > f(0) = 0$, $t > \tanh t$ and $t\cosh t > \sinh t$ for every $t > 0$.
Since the function $g(t)=t/\sinh t$ satisfies
$$
g'(t)=\frac{\sinh t - t\cosh t}{\sinh\!^2 t} < 0
$$
for every $t > 0$, the function $g$ is decreasing on $(0,\infty)$.
Consequently, $g(t) < \lim_{x\to 0} g(x) = 1$ for every $t > 0$, and
$$
L_S(\eta)
= l \, \frac{\sinh \e}{\sinh (l/2)}
< 2 \sinh \e
< 2 .
$$

Assume now that $C$ is a collar of a cusp $r$.
Lemma \ref{l:0} gives that $C=C(r,\l) \subset C(r,2)$ and $\l=2 \sinh \e$.
Thus, $L_S(\eta) = L_S(\p C) = \l = 2 \sinh \e < 2$.
\end{proof}

\begin{lemma} \label{l:area}
Let $S$ be a non-exceptional Riemann surface, $0<\e < \Arcsinh 1$ and $C$ a connected component of $\{z\in S\,| \; \inj(z)< \e \}$ that is a collar of a simple closed geodesic $\g$.
Then
$$
A_S( C) = 2 L_S(\g) \, \sqrt{\frac{\sinh\!^2 \e}{\sinh\!^2 (L_S(\g)/2)}-1}
< 4 \sinh \e.
$$
Furthermore, if $L_S(\g)/2 \le \d <\e$, then
$$
A_S( C)
\ge \, \frac{4 \d}{\sinh \d} \, \sqrt{\sinh\!^2 \e-\sinh\!^2 \d} \,.
$$
\end{lemma}

\begin{proof}
Let $l:=L_S(\g)$.
Lemma \ref{l:0} gives $C=C(\g,h) \subset C(\g)$ and $h=\Arccosh \frac{\sinh \e}{\sinh (l/2)}$.
It is well-known that $A_S( C) = 2 l \sinh h$, and so,
$$
\begin{aligned}
A_S( C)
& = 2 l \sinh h
= 2 l \, \sqrt{\cosh\!^2 h -1 }
= 2 l \, \sqrt{\frac{\sinh\!^2 \e}{\sinh\!^2 (l/2)}-1}
\\
& = \frac{2 l}{\sinh (l/2)} \, \sqrt{\sinh\!^2 \e-\sinh\!^2 (l/2)}
\,.
\end{aligned}
$$
Since the function $g(t)=t/\sinh t$ is decreasing on $(0,\infty)$, we have
$$
A_S( C)
< 4 \, \sqrt{\sinh\!^2 \e-\sinh\!^2 (l/2)}
< 4 \sinh \e .
$$
If $l/2 \le \d <\e$, then
$$
A_S( C)
\ge \frac{4 \d}{\sinh \d} \, \sqrt{\sinh\!^2 \e-\sinh\!^2 \d}
\,.
$$
\end{proof}

\begin{lemma} \label{l:area2}
Let $S$ be a non-exceptional Riemann surface, $0 < \e < \Arcsinh 1$ and $C$ a connected component of $\{z\in S\,| \; \inj(z)< \e \}$ that is a collar of a simple closed geodesic $\g$ with
$$
\d_0 \le \log\frac43 \,,
\qquad
L_S(\g)\le 2 \Arcsinh \Big( \frac{\sqrt3}{4}\,\sinh \e \Big) =: 2\d ,
\qquad
h=\Arccosh \frac{\sinh \e}{\sinh (L_S(\g)/2)}\,.
$$
Then, $\d_0 < h$ and
$$
A_S \big( C(\g,h-\d_0)\big)
>  \frac{1}{2}A_S \big(C(\g,h)\big)
\ge \frac{2\d}{\sinh \d} \, \sqrt{\sinh\!^2 \e-\sinh\!^2 \d} \,.
$$
\end{lemma}

\begin{proof}
By Lemma \ref{l:0} we have $C=C(\g,h)$.
Define $h'=h-\d_0$.
We also have
$$
\begin{aligned}
& L_S(\g) \le 2 \Arcsinh \Big( \frac{\sqrt3}{4}\,\sinh \e \Big)
\quad \Leftrightarrow \quad
\frac{4}{\sqrt3}\le \frac{\sinh \e}{\sinh (L_S(\g)/2)}
\\
& \quad \Leftrightarrow \quad
\log \frac{16}{3} \le 2 \log \frac{\sinh \e}{\sinh (L_S(\g)/2)}
\quad \Rightarrow \quad
\log \frac{16}{3} < 2 \Arccosh \frac{\sinh \e}{\sinh (L_S(\g)/2)} = 2 h
\\
& \quad \Rightarrow \quad
\d_0 \le \log\frac43 < 2h- \log{4}
\quad \Rightarrow \quad
e^{\d_0} <\frac{e^{2h}}{4}
\\
& \quad \Leftrightarrow \quad
e^{-h'}=e^{\d_0 -h}<\frac{e^h}{4}
\quad \Rightarrow \quad
e^{h'}-e^{-h'}> e^{h'} - \frac{e^h}{4}\,.
\end{aligned}
$$
Besides,
$\d_0 \le \frac12 \log\frac{16}{9} < \frac12 \log\frac{16}{3} < h$,
and $h'=h-\d_0$.
We have
$$
\begin{aligned}
& \d_0 \le \log \frac{4}{3}
\quad \Leftrightarrow \quad
\frac{e^h}{e^{\d_0}} \ge \frac{3}{4}\, e^h
\\
& \quad \Leftrightarrow \quad
e^{h'} \ge \frac{3}{4}\, e^h
\quad \Rightarrow \quad
2\Big(e^{h'}-  \frac{e^h}{4}\Big) \ge e^h > e^h-e^{-h}.
\end{aligned}
$$
Denote by $C'$ the collar $C'=C(\g,h')$.
Since $A_S(C')=2L_S(\g)\sinh h'$ and $A_S(C)=2L_S(\g)\sinh h$, we conclude
$$
\begin{aligned}
A_S(C') & =2L_S(\g)\sinh h'
= L_S(\g) \big( e^{h'}-e^{-h'} \big)
> L_S(\g) \Big(e^{h'}-  \frac{e^h}{4}\Big)
\\
& > \frac12\, L_S(\g) \big( e^{h}-e^{-h} \big)
= \frac12\, 2L_S(\g)\sinh h
= \frac12\, A_S(C).
\end{aligned}
$$
\indent
Finally, by Lemma \ref{l:area}, since
$$
L_S(\g)/2 \le \d = \Arcsinh \Big( \frac{\sqrt3}{4}\,\sinh \e \Big) < \e ,
$$
we have
$$A_S(C') > \frac{1}{2}A_S (C)\ge \frac{2\d}{\sinh \d} \, \sqrt{\sinh\!^2 \e-\sinh\!^2 \d}\,.$$
\end{proof}

%

\begin{lemma} \label{l:us}
Let $S$ be a non-exceptional Riemann surface, $0<\e < \Arcsinh 1$ and $C_1,C_2$ two different connected components of $\{z\in S\,| \; \inj(z)< \e \}$.
Let $K_1$ be the collar given by the Collar Lemma corresponding to $C_1$.
Then
$$
d_S(C_1,\p K_1) \ge \log \frac1{\sinh \e}\,,
\qquad
d_S(C_1,C_2) \ge 2\log \frac1{\sinh \e}\,.
$$
Furthermore, if $C_1$ is a collar of a simple closed geodesic $\g$, then
$d_S(C_1,\p K_1)$ is an increasing function on $L_S(\g)$.
\end{lemma}


\begin{proof}
By Margulis Lemma, $C_1$ and $C_2$ are collars.

Assume that $C_1$ is a collar $C(\g,h)$ for some simple closed geodesic $\g$.
Collar Lemma gives that there exists a collar $C(\g)=K_1$ about $\g$ of width $w$, where $\cosh w = \coth (L_S(\g)/2)$.
Lemma \ref{l:0} gives that $C_1$ is contained in $K_1$ and
$$
d_S( C_1, \p K_1)
= d_S(\p C_1, \p C(\g))
= w - h
= \Arccosh \frac{\cosh (L_S(\g)/2)}{\sinh (L_S(\g)/2)} - \Arccosh \frac{\sinh \e}{\sinh (L_S(\g)/2)}
\,.
$$

If $0 < t < \e$, then
$$
\begin{aligned}
\sinh\!^2 \e \cosh\!^2 t
& > \sinh\!^2 \e - \sinh\!^2 t,
\qquad
\frac{\sinh \e \cosh t}{\sqrt{\sinh\!^2 \e - \sinh\!^2 t}} > 1,
\\
& \frac{\sinh \e \cosh t}{\sinh t \, \sqrt{\sinh\!^2 \e - \sinh\!^2 t}} - \frac{1}{\sinh t}
> 0.
\end{aligned}
$$
If $0 < t < \e$, then we define the function
$$
f(t)
= \Arccosh \frac{\cosh t}{\sinh t} - \Arccosh \frac{\sinh \e}{\sinh t}
\,.
$$
Since
$$
f'(t)
= \frac{\sinh \e \cosh t}{\sinh t \, \sqrt{\sinh\!^2 \e - \sinh\!^2 t}} - \frac{1}{\sinh t}
>0,
$$
$d_S(C_1,\p K_1)$ is an increasing function in $L_S(\g)$.
In particular, $f(t)> \lim_{x \to 0^+} f(x)$ for every $0 < t < \e$.
Since
$$
\Arccosh x - \Arccosh y
= \Arccosh \big( xy - \sqrt{x^2-1} \sqrt{y^2-1} \,\big)
,
$$
we consider
$$
\begin{aligned}
\frac{\cosh t}{\sinh t} \frac{\sinh \e}{\sinh t}
& - \sqrt{\frac{\cosh\!^2 t}{\sinh\!^2 t}-1} \, \sqrt{\frac{\sinh\!^2 \e}{\sinh\!^2 t}-1}
\\
& = \frac{\cosh t \sinh \e - \sqrt{\sinh\!^2 \e-\sinh\!^2 t}}{\sinh\!^2 t}
\, \frac{\cosh t \sinh \e + \sqrt{\sinh\!^2 \e-\sinh\!^2 t}}{\cosh t \sinh \e + \sqrt{\sinh\!^2 \e-\sinh\!^2 t}}
\\
& = \frac{\sinh\!^2 \e + 1}{\cosh t \sinh \e + \sqrt{\sinh\!^2 \e-\sinh\!^2 t}}
\,.
\end{aligned}
$$
Thus,
$$
\begin{aligned}
\lim_{t \to 0^+} & \Big( \,\frac{\cosh t}{\sinh t} \frac{\sinh \e}{\sinh t}
- \sqrt{\frac{\cosh\!^2 t}{\sinh\!^2 t}-1} \, \sqrt{\frac{\sinh\!^2 \e}{\sinh\!^2 t}-1}\; \Big)
\\
& = \lim_{t \to 0^+} \frac{\sinh\!^2 \e + 1}{\cosh t \sinh \e + \sqrt{\sinh\!^2 \e-\sinh\!^2 t}}
= \frac{\sinh\!^2 \e + 1}{2 \sinh \e}
\,,
\\
\lim_{t \to 0^+} f(t)
& = \Arccosh \frac{\sinh\!^2 \e + 1}{2 \sinh \e}
= \log \Big( \,\frac{\sinh\!^2 \e + 1}{2 \sinh \e} + \sqrt{\frac{(\sinh\!^2 \e + 1)^2}{4 \sinh\!^2 \e}-1} \; \Big)
\\
& = \log \Big( \,\frac{\sinh\!^2 \e + 1}{2 \sinh \e} + \frac{1-\sinh\!^2 \e}{2 \sinh \e} \; \Big)
= \log \frac1{\sinh \e}
\,.
\end{aligned}
$$
Hence,
$$
d_S( C_1, \p K_1)
= \Arccosh \frac{\cosh (L_S(\g)/2)}{\sinh (L_S(\g)/2)} - \Arccosh \frac{\sinh \e}{\sinh (L_S(\g)/2)}
>\log \frac1{\sinh \e} \,.
$$

Assume now that $C_1$ is a collar $C(r,\l)$ for some cusp $r$.
Consider a fundamental domain for $S$ in the upper half-plane $\HH$
contained in $\{\zeta\in\HH: \, 0\le \Re \zeta \le 1\}$
and such that $\{\zeta\in\HH: \, 0\le \Re \zeta \le 1,\,\Im \zeta > 1/2\}$
corresponds to $C(r,2)=K_1$.
Thus, $\{\zeta\in\HH: \, 0\le \Re \zeta \le 1,\,\Im \zeta > 1/\l\}$ corresponds to $C_1$, and Lemma \ref{l:0} gives that
$\l = 2\sinh \e < 2$
and $C_1$ is contained in $C(r,2)$.
Hence,
$$
d_S( C_1, \p K_1)
= d_S(\p C_1, \p C(r,2))
= d_\HH(i/\l, i/2)
= \log \frac2{\l}
= \log \frac1{\sinh \e}
\,.
$$
\indent
Since collars of geodesics and cusps are pairwise disjoint by Remark \ref{r:CollarLemma}, we conclude
$$
d_S(C_1,C_2) \ge 2\log \frac1{\sinh \e}\,.
$$
\end{proof}

Let $X$ be a non-exceptional Riemann surface and $0<\e < \Arcsinh 1$.
Consider the thick-thin decomposition given by Margulis Lemma:
$X_\e = \{z\in X\,| \; \inj(z) \ge \e \}$ and $X \setminus X_\e$.

Given the cusps $\{r_i\}_{i\in I}$ in $X$, let us consider for each $r_i$ the collar in $X$ about $r_i$ with area $2\sinh \e$, $C_i=C(r_i,2\sinh \e)$, and let $\s_i=\partial C(r_i,2\sinh \e)$.

Given the simple closed geodesics $\{\g_j\}_{j\in J}$ in $X$ with $L_X(\g_j) < 2 \e$,
let us consider for each $\g_j$ its collar $K_j$ in $X$ of width $h_j=\Arccosh \frac{\sinh \e}{\sinh (L_X(\g_j)/2)}$.

By Lemma \ref{l:0}, we have
\begin{equation}\label{eq:0}
X\setminus X_\e
= \big( \cup_{i\in I} C_i\big) \cup \big(\cup_{j\in J} K_j\big).
\end{equation}

The following result is well-known.

\begin{lemma} \label{l:circulosycircunferencias}
Let $S$ be a non-exceptional Riemann surface, $z\in S$ and $r > 0$.
Then
$$
A_S \big( B_r(z) \big) \le 4\pi \sinh\!^2 \frac{r}{2}\,,
\qquad
L_S \big( \p B_r(z) \big) \le 2\pi \sinh r .
$$
Furthermore, if $0<r \le \inj(z)$, then
$$
A_S \big( B_r(z) \big) = 4\pi \sinh\!^2 \frac{r}{2}\,,
\qquad
L_S \big( \p B_r(z) \big) = 2\pi \sinh r .
$$
\end{lemma}

%

\section{Surfaces and graphs}

We are going to construct a graph associated to any non-exceptional Riemann surface.

\smallskip



A subset $A$ in a metric space $(X,d)$ is called \emph{$r$-separated},
$r>0$, if $d(a,a')\geq r$ for any distinct $a,a'\in A$. Note that
if $A$ is maximal with this property, then the union $\cup_{a\in
A} B_r(a)$ covers $X$. A maximal $r$-separated set $A$ in a metric
space $X$ is called an $r$-\emph{approximation} of $X$.

Fix $0 < \varepsilon < \Arcsinh 1$.
Given any $\varepsilon$-approximation $A_\varepsilon$ of $X$,
the graph $\Gamma_{A_\varepsilon}=(V,E)$ with $V=A_\varepsilon$ and $E:=\{xy \, | \, x,y \in A_\varepsilon \mbox{ with } 0<d(x,y)\leq 2\varepsilon\}$ is called an $\varepsilon$-\emph{net}.

Given any $\d>0$ satisfying
\begin{equation} \label{eq:d1}
\d< \d_1:=\min\Big\{ \log \frac{1}{\sinh \e}\,, \, \Arcsinh \Big( \frac{\sqrt3}{4}\,\sinh \e \Big) \Big\} ,
\end{equation}
let $J_\d:=\{j \in J \, | \, L(\gamma_j)< 2\d\}$, where $J$ denotes the index set of the collars about simple closed geodesics in the thick-thin decomposition from equation \eqref{eq:0}.
Note that $\d_1 < \e$ and so, $\d < \e$.
Let $[X_\e]_\d=X_\e\cup \big(\cup_{j\in J\setminus J_\d} K_j\big)$.
Therefore, $[X_\e]_\d$ is the surface $X$ without the collars $C_i=C(r_i, 2 \sinh \e)$ about the cusps and the collars $K_j=C(\g_j,h_j)$ with $L_S(\g_j)<2\d$. Notice that $X_\e \subseteq [X_\e]_\d \subseteq X_\d$.
Let $\{[X_\e]_\d^r\}_r$ be the connected components of $[X_\e]_\d$.
For each $r$, let $[A]_\d^r$ be a $\d$-approximation of $[X_\e]_\d^r$ (with respect to the distance in $X$)
and $\Gamma_{[A]_\d^r}$ the corresponding $\delta$-net.
Let $[A]_\d= \cup_r [A]_\d^r$ and $\Gamma_{[A]_\d}= \cup_r \Gamma_{[A]_\d^r}$.
To simplify the notation we make no distinction between the vertices of the graph and the corresponding points in the surface.
For each cusp $r_i$ let us define a vertex $w_i$ and for each $\gamma_j$ with $j\in J_\d$ let us define a new vertex $v_j$.
Let us define a graph $\Gamma_\delta(X)$ with vertices $[A]_\d\cup \{w_i\}_{i\in I} \cup \{v_j\}_{j\in J_\d}$, containing $\Gamma_{[A]_\d}$ as a subgraph and such that $w_i$ is adjacent to every vertex $v$ in $[A]_\d$ such that $d(v,C(r_i,2\sinh \e))\leq \d$ and $v_j$ is adjacent to every vertex $v$ in $[A]_\d$ such that $d(v,K_j)\leq \d$.

\begin{remark} \label{r:117}
From the definition of $\Gamma_{\d}(X)$, it follows that no pair of vertices in $\{w_i\}_{i\in I} \cup \{v_j\}_{j\in J_\d}$ can be adjacent in $\Gamma_\delta(X)$.
By Lemma \ref{l:us}, $\{N(w_i)\}_{i\in I} \cup \{N(v_j)\}_{j\in J_\d}$ are pairwise disjoint sets in $\Gamma_\delta(X)$.
\end{remark}

We have the following result in \cite[Lemma 2.3]{K}.

\begin{lemma} \label{l: growth0}
Let $X$ be a complete Riemannian $n$-manifold whose Ricci curvature
is bounded from below by $-(n-1)K^{2}$ $(K>0)$, and let $P$ be an $\varepsilon$-separated subset of $X$.
Then we have
$|\{p\in P:x\in B_{r}(p)\}| \le \nu$
for all $r>0$ and for all $x\in X$, where $\nu=\nu(n, K, \varepsilon, r)>0$. Consequently, every
$\varepsilon$-net in a complete Riemannian manifold whose Ricci curvature is bounded from
below is uniform.
\end{lemma}

Lemma \ref{l: growth0} has the following consequence.

\begin{lemma} \label{l: growth}
If $X$ is a non-exceptional Riemann surface and $[A]_\d$ is a  $\d$-separated set in $[X_\e]_\d$, then
there is some constant $\mu(\d)$ such that $|B_{3\d}(p)\cap [A]_\d|\leq \mu(\d)$ for every $p\in [X_\e]_\d$.
\end{lemma}

\begin{proposition} \label{p:uniform}
If $X$ is a non-exceptional Riemann surface,
$0 < \varepsilon < \Arcsinh 1$ and $0<\d<\d_1$,
then $\Gamma_{\delta}(X)$ is uniform.
\end{proposition}

\begin{proof} Let $v$ be any vertex in $\Gamma_{\delta}(X)$.
If $v\in [A]_\d$, then its adjacent vertices $w$ in $[A]_\d$ satisfy $d(v,w) \le 2\d$ and, by Lemma \ref{l: growth}, these are at most $\mu(\d)$.
Also, by Remark \ref{r:117}, $v$ has at most one adjacent vertex $v_j$ or $w_i$,
corresponding to a collar $K_j$ or $C_i$, respectively.

If $v=w_i$ is the vertex corresponding to a collar $C_i=C(r_i,2\sinh \e)$ then, by Lemma \ref{l:dc}, $\partial C_i =\s_i$ satisfies $L(\s_i) \le 2\sinh \e<2$.
Let $\{x_1,\dots,x_n\}$ be a maximal $\d$-separated set in $\s_i$.
Then, $n\leq \frac{2\sinh \e}{\d}$ and for any vertex $w$ adjacent to $v$, we have $w\in [A]_\d$,
$d(w,\s_i) \le \d$ and so, $w$ must be contained in some $B_{2\d}(x_k)$.
Thus, by Lemma \ref{l: growth}, there are at most $\frac{2\sinh \e}{\d}\mu(\d)<\frac{2}{\d}\mu(\d)$ vertices adjacent to $v$ in $\Gamma_{\delta}(X)$.

If $v=v_j$ is the vertex corresponding to a collar $K_j=C(\g_j,h)$ then, by Lemma \ref{l:dc},
$\partial K_j =\eta^j_1\cup \eta^j_2$ and $L(\eta^j_i) \le 2\sinh \e<2$.
Let $\{x_1,\dots,x_n\}$ be a maximal $\d$-separated set in $\eta^j_1\cup \eta^j_2$.
Then, $n\leq \frac{4\sinh \e}{\d}$ and for any vertex $w$ adjacent to $v$, we have $w\in [A]_\d$,
$d(w, \eta_1^j \cup \eta_2^j) \le \d$ and so, $w$ must be contained in some $B_{2\d}(x_k)$.
Thus, by Lemma \ref{l: growth}, there are at most $\frac{4\sinh \e}{\d}\mu(\d)<\frac{4}{\d}\mu(\d)$ vertices adjacent to $v$ in $\Gamma_{\delta}(X)$.
\end{proof}

Let $X$ be a non-exceptional Riemann surface
and $\Gamma_\delta(X)$ the associated graph with vertex set $[A]_\d\cup \{w_i\}_{i\in I} \cup \{v_j\}_{j\in J_\d}$.
Given a geodesic domain $G$ in $X$, let $S_G=([A]_\d\cap G)\cup \{w_i : \, C(r_i,2\sinh \e)\subset G \} \cup \{v_j : \, C(\g_j,h_j)\cap G \neq \emptyset\}$.

\begin{remark}\label{r: boundary}
Let $X$ be a non-exceptional Riemann surface and $\Gamma_\delta(X)$ the associated graph defined above.
Given a geodesic domain $G$, then $\partial S_G$ is the set of vertices $w\in [A]_\d\setminus G$ satisfying one of the following conditions:
\begin{itemize}
	\item[a)] $w$ is adjacent to a vertex $v\in [A]_\d\cap G$ and either $\bar{B}_\d(v)\cap \partial G \neq \emptyset$ or $\bar{B}_\d(w)\cap \partial G \neq \emptyset$.
In particular, we have either $\bar{B}_\d(v)\cap \partial_{2\d} G \neq \emptyset$ or $\bar{B}_\d(w)\cap \partial_{2\d} G \neq \emptyset$.
	\item[b)] $w$ is adjacent to a vertex  $v\in S_G$ with $v=v_j$ for some $j\in J_\d$ and so, $C(\g_j,h_j)\cap \partial G \neq \emptyset$.
\end{itemize}

Notice that $w \neq v_j$ for every $j \in J_\d$ with $C(\g_j,h_j)\cap G =\emptyset$,
since, by Lemma \ref{l:us}, the distance from $C(\g_j,h_j)$ to $G$ is at least $\log \frac1{\sinh \e}> \d$.

Note also that, since the collars are pairwise disjoint, given a geodesic domain $G$ and a collar $C_i$, we have either $C_i \subset G$ or $C_i \cap G = \emptyset$.
\end{remark}

Let $\partial_{2\d} S_G$ denote the set of vertices $w\in [A]_\d\setminus G$  satisfying  condition $a)$ above.

\begin{remark} \label{O: dist}
For every vertex $w\in \partial_{2\d} S_G$, $d(w,\p_{2\d} G)\leq 2\d$.
\end{remark}

\begin{theorem} \label{th:graph}
Given a $(2\d_2)$-regular non-exceptional Riemann surface $X$,
$0 < \varepsilon < \Arcsinh 1$ and $0<\d<\min\{\d_1,\d_2\}$,
$X$ satisfies LII if and only if $\Gamma_{\delta}(X)$ satisfies LII.
\end{theorem}

\begin{proof} Suppose $X$ satisfies LII. Consider any non-empty finite subset of vertices $S$ in $\Gamma_{\delta}(X)$. Let us define
$$
\Omega= \Big(\cup_{v\in S\cap [X_\e]_\d}B_{\delta}(v)\Big) \cup
\Big(\cup_{\{ i\in I \,: \,w_i\in S \}}C_i\Big)
\cup \Big(\cup_{\{ j\in J_\d \,: \,v_j\in S \}}K_j\Big).
$$

Choose $\d'$ satisfying
$$
\d'< \min \Big\{ \d, \, 2 \log\frac43 \,\Big\} < \d_1 < 2\Arcsinh \Big( \frac{\sqrt3}{4}\,\sinh \e \Big) .
$$
Then we have the following:
\begin{itemize}
	\item for any pair of vertices $v_1,v_2\in S\cap [X_\e]_\d$,
$B_{\d'\!/2}(v_1)\cap B_{\d'\!/2}(v_2)=\emptyset$,
	\item for any vertex $v\in S\cap [X_\e]_\d$ and any $i\in I$, we have $B_{\d'\!/2}(v)\cap C(r_i,\lambda')=\emptyset$,
where $\lambda'= e^{-\d'\!/2} 2 \sinh \e$,
	\item for any vertex $v\in S\cap [X_\e]_\d$ and any $j\in J_\d$, we have
$B_{\d'\!/2}(v)\cap C(\g_j,h_j-\d'\!/2)=\emptyset$.
\end{itemize}
Let $A'=\frac{\d'}{\sinh (\d'\!/2)}\sqrt{\sinh^2\!\e -\sinh^2 (\d'\!/2)}$ and
$K=\min\{4\pi \sinh^2(\d'\!/4), \lambda',A'\} $.
Lemmas \ref{l:area2} (with $\d_0 = \d'/2$) and \ref{l:circulosycircunferencias} give
$\d'\!/2<h_j$ and $A(C(\g_j,h_j-\d'/2)) > A'$ for every $j\in J_\d$, and
\[K\, | S | \leq 4\pi \sinh^2\frac{\d'}{4}\, \big|S\cap [X_\e]_\d\big| + \lambda' \, \big|S\cap \{w_i\}_{i\in I} \big| + A' \, \big|S\cap \{v_j\}_{j\in J_\d} \big| \leq A(\Omega).\]

Now, let $\p'\!S=\partial (\G_\d(X) \backslash S)$.
By Lemmas \ref{l:circulosycircunferencias} and \ref{l:dc},
$$
\begin{aligned}
L(\partial{\Omega})
& \leq  \sum_{v\in \p'\!S\cap [X_\e]_\d} L(\partial B_\d(v))
+ \sum_{\{ i\in I \,:\,w_i \in \p'\!S \}} L(\p C_i)
+ \sum_{\{ j\in J_\d \,:\,v_j \in \p'\!S \}} L(\p K_j)
\\
& \leq 2\pi \sinh \d \, \big|\p'\!S\cap [X_\e]_\d\big|
+ 2\sinh \e \, \big|\p'\!S\cap \{w_i\}_{i\in I} \big|
+ 4\sinh \e \, \big|\p'\!S\cap \{v_j\}_{j\in J_\d} \big|
\\
& \leq \max\{2\pi \sinh \d,4\sinh \e \}\, | \p'\!S|.
\end{aligned}
$$
Since, by Proposition \ref{p:uniform}, $\Gamma_{\delta}(X)$ is $\mu$-uniform for some constant $\mu >0$, we have $| \p'\! S|\leq \mu \, |\partial S|$.
Therefore, \[L(\partial \Omega)\leq \max\{2\pi \sinh \d,4\sinh \e \} \,\mu \, |\partial S|.\]

Thus, since $X$ satisfies LII, there is some constant $h$ such that $A(\Omega)\leq h \, L(\partial \Omega)$ and, hence,
\[| S| \leq \frac{1}{K} \, A(\Omega) \leq \frac{h}{K} \, L(\partial \Omega)\leq \frac{h \mu}{K}  \max\{2\pi \sinh \d,4\sinh \e \}  \, | \partial S|, \]
and $\Gamma_{\delta}(X)$ satisfies LII.

\smallskip

Assume now that $\Gamma_{\delta}(X)$ satisfies LII. By Lemma  \ref{l:FR1}, it suffices to check that $X$ satisfies LII on every geodesic domain. Let $G$ be any geodesic domain and consider the set $S_G\subset V(\Gamma_{\delta}(X))$.
Then, since
$$
G\subset \Big(\cup_{v\in S_G\cap [X_\e]_\d}B_{2\d}(v)\Big)
\cup \Big( \cup_{\{ i\in I \,:\,w_i \in S_G \}} C_i \Big)
\cup \Big( \cup_{\{ j\in J_\d \,:\,v_j \in S_G \}} K_j \Big),
$$
we have, by Lemmas \ref{l:circulosycircunferencias}, \ref{l:0} and \ref{l:area},
\[
\begin{aligned}
A(G)
& \leq 4\pi \sinh^2 \!\d\, \big|S_G\cap [X_\e]_\d\big| + 2\sinh \e \big|S_G\cap \{w_i : \, i\in I\}\big| +
4\sinh \e \big|S_G\cap \{v_j : \, j\in J_\d\}\big|
\\
& \leq \max\{4\pi \sinh^2\!\d, 4\sinh \e \} \, |S_G|.
\end{aligned}
\]

For any connected component $\gamma$ of $\partial_{2\d} G$ consider a maximal $\d$-separated set $\{x_1,\dots,x_n\}$. Notice that $n\leq \frac{L(\g)}{\d}$.
By Remark \ref{O: dist}, for every point $v\in \partial_{2\d} S_G$, $d(v,\p_{2\d} G)\leq 2\d$ and since $\{x_1,\dots,x_n\}$ is a maximal $\d$-separated set, there is some $x_i$ such that $d(v,x_i)<3\d$.

Therefore, by Lemma \ref{l: growth}, there exists a constant $\mu(\d)$ such that
\begin{equation}\label{eq:G1} |\partial_{2\d} S_G| \leq \sum_{\g \subseteq \partial_{2\d} G}\mu(\d)\frac{L(\g)}{\d}=
\frac{\mu(\d)}{\d} L(\partial_{2\d} G).
\end{equation}

Also, by Remark \ref{r: boundary}, if $w\in \partial S_G \setminus \partial_{2\d} S_G$, then $w$ is adjacent to some $v_j$ with $j\in J_\d$ and, by Proposition \ref{p:uniform}, it follows that for every such $v_j$ there exist at most $\D$ vertices adjacent to it. Hence,
\begin{equation}\label{eq:G2}
 |\partial S_G\setminus \partial_{2\d} S_G|\leq \D CC(\partial G \setminus \partial_{2\d}G).
\end{equation}

Since $X$ is $(2\d_2)$-regular and $0<\d<\d_2$, $X$ is also $(2\d)$-regular and there exists $c>0$ such that $L(\partial_{2\d} G) \ge c \,CC(\partial G \setminus \partial_{2\d}G)$.
Therefore, by inequalities \eqref{eq:G1} and \eqref{eq:G2}, it follows that
\[
\begin{aligned}
|\partial S_G|= & |\partial_{2\d} S_G|+ |\partial S_G\setminus \partial_{2\d} S_G|\leq \frac{\mu(\d)}{\d} L(\partial_{2\d} G) + \D CC(\partial G \setminus \partial_{2\d}G)\\
& \leq  \frac{\mu(\d)}{\d} L(\partial_{2\d} G) + \frac{\D}{c} L(\partial_{2\d}G)\leq  \Big(\frac{\mu(\d)}{\d} + \frac{\D}{c} \Big)L(\partial G).
\end{aligned}
\]

Thus, since $\Gamma_{\delta}(X)$ satisfies LII, there is some constant $h$ such that $| S_G| \leq h \, |\partial S_G|$, and so,
\[
\begin{aligned}
A(G) & \leq  \max\{4\pi \sinh^2\!\d, 4\sinh \e\}\, | S_G|
\leq \max\{4\pi \sinh^2\!\d, 4\sinh \e\} h \, |\partial S_G| \\
& \leq \max\{4\pi \sinh^2\!\d, 4\sinh \e\} \Big(\frac{h\mu(\d)}{\d} + \frac{h\D}{c} \Big) L(\partial G).
\end{aligned}
\]
Hence, $X$ satisfies LII on every geodesic domain, finishing the proof.
\end{proof}

\begin{remark}
Note that the hypothesis of regularity in Theorem \ref{th:graph} is not restrictive at all,
since Theorem \ref{t:regular} gives that non-regular surfaces do not have LII.
\end{remark}

\section{Gromov boundary and LII}

We present in this section an application of Theorem \ref{th:graph} relating Gromov boundary and LII (see Theorem \ref{th:sufic}).
First of all, we need some background on Gromov hyperbolicity.

Let $X$ be a metric space. Fix a base point $o\in X$ and for
$x,x'\in X$ let
$$(x|x')_o=\frac{1}{2}\big(d(x,o)+d(x',o)-d(x,x')\big).$$
The number $(x|x')_o$ is non-negative and it is called the \emph{Gromov product} of $x,x'$ with respect to $o$.

\begin{definition} A metric space $X$ is \emph{(Gromov)
hyperbolic} if it satisfies the $\delta$-inequality
\[(x|y)_o\geq \min\{(x|z)_o,(z|y)_o\}-\delta\] for some $\delta\geq
0$, for every base point $o\in X$ and all $x,y,z \in X$.
\end{definition}


If $X$ is a metric space and $I\subset \RR$ is an interval, we say that the curve $\g:I\longrightarrow X$ is a \emph{geodesic} if we have $L(\g|_{[t,s]})=d(\g(t),\g(s))=|t-s|$ for every $s,t\in [a,b]$
(then $\gamma$ is equipped with an arc-length parametrization).
A \emph{geodesic ray} is a geodesic defined on the interval $[0,\infty)$.

A \emph{geodesic metric space} is a metric space such that for every couple of points there exists a geodesic joining them.

Let us recall the following from \cite{MR}.

\begin{definition}  A geodesic metric space $X$ has a \emph{pole} in a point $v$ if there
exists $M > 0$ such that each point of $X$ lies in an $M$-neighborhood of some geodesic
ray emanating from $v$.
\end{definition}

Let $X$ be a hyperbolic space and $o\in X$ a base point.

The \emph{relative geodesic boundary} of $X$ with respect to the base point $o$ is the set of equivalence classes
\[
\partial_o^g X := \{ [\gamma]  \, | \, \gamma: [0,\infty) \to X \mbox{ is a geodesic ray with } \gamma(0) = o\},
\]
where two geodesic rays $\gamma_1, \gamma_2$ are equivalent if there exists some $K>0$ such that
$d(\gamma_1(t),\gamma_2(t))<K$, for every $t\geq 0.$

In fact, the definition above is independent from the base point.
Therefore, the set of classes of geodesic rays is called \emph{geodesic boundary} of $X$, $\partial^gX$.  Herein, we do not distinguish between the geodesic ray and its image.

A sequence of points $\{x_i\}\subset X$ \emph{converges to infinity}
if \[\lim_{i,j\to \infty} (x_i|x_j)_o=\infty.\] This property is
independent of the choice of $o$ since
\[|(x|x')_o-(x|x')_{o'}|\leq d(o,o')\] for any $x,x',o,o' \in X$.

Two sequences $\{x_i\},\{x'_i\}$ that converge to infinity are
\emph{equivalent} if \[\lim_{i\to \infty} (x_i|x'_i)_o=\infty.\]
Using the $\delta$-inequality, we easily see that this defines an
equivalence relation for sequences in $X$ converging to infinity.
The \emph{sequential boundary at infinity} $\partial_\infty X$ of $X$ is
defined to be the set of equivalence classes of sequences
converging to infinity.

Note that given a geodesic ray $\gamma$, the sequence $\{\gamma(n)\}$ converges to infinity and two equivalent rays induce equivalent sequences.
Thus, in general, $\partial^g X\subseteq \partial_\infty X$.

We say that a metric space is \emph{proper} if every closed ball is compact.
Every uniform graph and every complete Riemannian manifold are proper geodesic metric spaces.

\begin{proposition}\cite[Chapter III.H, Proposition 3.1]{BH}\label{Prop: equiv_boundary}
If $X$ is a proper hyperbolic geodesic metric space, then the natural map from $\partial^g X$ to $\partial_\infty X$ is a bijection.
\end{proposition}

For every $\xi, \xi' \in \partial_\infty X$, its Gromov product
with respect to the base point
$o\in X$ is defined as
\[ (\xi|\xi')_o =  \inf \ \liminf_{i\to \infty} (x_i|x'_i)_o,\]
where the infimum is taken over all sequences $\{x_i\} \in \xi $, $\{x'_i\} \in \xi' $.

A metric $d$ on the sequential boundary at infinity
$\partial_\infty X$ of $X$ is said to be \emph{visual}, if there are $o\in X$, $a > 1$ and positive
constants $c_1$, $c_2$, such that
\[c_1a^{-(\xi|\xi')_o} \leq  d(\xi, \xi') \leq c_2a^{-(\xi|\xi')_o}\]
for all $\xi, \xi' \in \partial_\infty X$. In this case, we say that $d$ is a visual metric
with respect to the base point $o$ and the parameter $a$.

\begin{theorem}\cite[Theorem 2.2.7]{BS} Let $X$ be a hyperbolic space. Then for any $o \in X$,
there is $a_0 > 1$ such that for every $a \in (1, a_0]$ there exists a metric $d$ on
$\partial_\infty X$, which is visual with respect to $o$ and $a$.
\end{theorem}

\begin{remark} Notice that for any visual metric, $\partial_\infty X$ is bounded and complete.
\end{remark}

\begin{definition} Given a metric space $(X,d)$ and a constant $S>1$, $(X,d)$ is \emph{$S$-uniformly perfect} if there exists some $\varepsilon_0>0$ such that for every $x\in X$ and every $0<\varepsilon \le \varepsilon_0$ there exists a point $y\in X$ such that
$\frac{\varepsilon}{S} <d(x,y) \le \varepsilon$. $(X,d)$ is \emph{uniformly perfect} if there exists some $S$ such that $(X,d)$ is $S$-uniformly perfect.
\end{definition}

\begin{theorem}\cite[Theorem 4.15]{MR} \label{t:iigraph}
Given a hyperbolic uniform graph $\Gamma$ with a pole, then $\Gamma$ has LII if and only if $\partial_\infty \Gamma$ is uniformly perfect for some visual metric.
\end{theorem}

\smallskip

Let $X$ be a regular non-exceptional Riemann surface and consider some constants $0<\d<\e$ such that Theorem \ref{th:graph} holds.
Then, for every collar $C(\g_j,h_j)$ of every simple closed geodesic $\g_j$ with $L(\g_j)<2\d$, consider $\partial C(\g_j,h_j)=\eta^j_1\cup \eta^j_2$ (where $\eta^j_i$ are Jordan curves) and let us define the (non-smooth) surface $[X]/\sim_\d$ where $C(\g_j,h_j)$ is removed, $\eta^j_1$ is glued along $\eta^j_2$ by identifying the intersection in both curves with the same pencil geodesic, and
the collars about the cusps $\{ C(r_i, 2 \sinh \e) \}_{i \in I}$ are removed.
Consider on $[X]/\sim_\d$ the \emph{inner distance} inherited from $X$:
$$
d_{[X]/\sim_\d} (x,y) = \inf \big\{ L(\g) : \, \g \text{ is a curve in } [X]/\sim_\d \text{ joining $x$ and } y \big\},
$$
where $L$ denotes the restriction to $[X]/\sim_\d$ of the Poincar\'e length in $X$.
One can check that the infimum in this definition is, in fact, a minimum, and that $[X]/\sim_\d$ is a proper geodesic metric space.

\begin{theorem} \label{Th:qi-graph}
Given a non-exceptional Riemannian surface $X$, $0 < \varepsilon < \Arcsinh 1$ and $0<\d< \d_1$, then $[X]/\sim_\d$ and $\G_\d(X)$ are quasi-isometric.
\end{theorem}

\begin{proof}
Let $\G_\d(X)/\sim_\d$ be the graph defined by $\G_\d(X)$ where vertices $\{w_i\}_{i\in I}$ and $\{v_j\}_{j\in J_\d}$ are removed and a new edge is defined between any pair of vertices adjacent to the same vertex $v_j$.
Notice that the vertex set of $\G_\d(X)/\sim_\d$ is $[A]_\d$.
Let us denote by $d_\G(v,w)$ the path distance in $\G_\d(X)/\sim_\d$ and by $d(v,w)$ the inner distance in $[X]/\sim_\d$.

Since $\{N(w_i)\}_{i\in I} \cup \{N(v_j)\}_{j\in J_\d}$
are pairwise disjoint sets by Remark \ref{r:117}, in order to prove
that $\G_\d(X)/\sim_\d$ is quasi-isometric to $\G_\d(X)$, it suffices to check that there is a constant $C>0$ such that given any two vertices, $u_1,u_2$, adjacent to some $v\in\{w_i\}_{i\in I}$ in $\G_\d(X)$, then $d_\G(u_1,u_2)<C$.


Claim: Given any pair of points $v,w\in [A]_\d$, $d_\G(v,w)<\mu(\d) \big(\frac{d(v,w)}{2\d}+2\big)$, where $\mu(\d)$ is the constant in Lemma \ref{l: growth}.
Consider any geodesic path $\g$ in $[X]/\sim_\d$ joining $v$ to $w$ and let $P_\g=\{u\in \G_\d(X)/\sim_\d \, \,: \, B_\d(v)\cap \g\neq \emptyset \}$.
It is clear that $\{B_\d(u) \, : \, u\in P_\g\}$ covers $\g$ and $d_\G(v,w)\leq |P_\g|$.
Denote by $\lceil t \rceil$ the upper integer part of $t$, i.e., the smallest integer greater than or equal to $t$.
Let $k= \big\lceil \frac{d(v,w)}{2\d}\big\rceil$ and $v=x_0,\dots, x_k=w$ in $\g$ such that $d(x_{i-1},x_i)=\frac{d(v,w)}{k}\leq 2\d$.
Notice that for every vertex $u\in P_\g$, $u\in B_{2\d}(x_i)$ for some $0\leq i\leq k$.
Hence, $d_\G(v,w)\leq |P_\g|\leq \mu(\d) (k+1)<\mu(\d) \big(\frac{d(v,w)}{2\d}+2\big)$.

If $u_1,u_2$ are adjacent to some $v\in\{w_i\}_{i\in I}$, it follows by Lemma \ref{l:dc} that $d(u_1,u_2) \le 2\d+\sinh \e$.
Then, by the claim above $d_\G(u_1,u_2)<\mu(\d) \big(\frac{d(u_1,u_2)}{2\d}+2\big) \le \mu(\d)\big(\frac{\sinh \e}{2\d}+3\big)$ and $\G_\d(X)/\sim_\d$ is quasi-isometric to $\G_\d(X)$.

Let us see now that $\G_\d(X)/\sim_\d$ is quasi-isometric to $[X]/\sim_\d$.
Let $h\colon [A]_\d \to [X]/\sim_\d$ be the inclusion map.
Now, consider any two vertices $v,w \in [A]_\d$ and let $v=u_0,\dots,u_k=w$ be the shortest path in $\G_\d(X)/\sim_\d$ joining them.
For $1\leq i \leq k$, if $u_{i-1}u_i$ is also an edge in $\G_\d(X)$, then $d(u_{i-1},u_i)\leq 2\d$.
If $u_{i-1}u_i \notin E(\G_\d(X))$, then
$u_{i-1}$ and $u_i$ are adjacent to the same vertex $v_j$ in $\G_\d(X)$; therefore, since
the length of each connected component of $\p K_j$ is at most $2\sinh \e$ by Lemma \ref{l:dc}, $d(u_{i-1},u_i)\leq 2\d+\sinh \e$.
Thus, $d(u_{i-1},u_i)\leq 2\d+\sinh \e$ for every $1\leq i \leq k$ and $d(v,w)\leq (2\d+\sinh \e)\, k= (2\d+\sinh \e) \, d_\G(v,w)$.

By the claim above,  $d_\G(v,w)<\mu(\d) \big(\frac{d(v,w)}{2\d}+2\big)$. Hence,
$$ \frac{2\d}{\mu(\d)} \, d_\G(v,w)-4\d \leq d(v,w) \leq (2\d+\sinh \e) \, d_\G(v,w)$$
and $h$ is a $(\max\{\frac{\mu(\d)}{2\d},2\d+\sinh \e\},4\d)$-quasi-isometric embedding. It is trivial to check that $h$ is $\d$-full.
\end{proof}

\begin{theorem}\cite[p.88]{GH}\label{th: stability_hyp}
If $f:X \rightarrow Y$ is a quasi-isometry between geodesic metric spaces, then $X$ is hyperbolic if and only if $Y$ is hyperbolic.
Furthermore, if $f$ is a $c$-full $(a,b)$-quasi-isometry and $X$ (respectively, $Y$) is $\d$-hyperbolic, then $Y$ (respectively, $X$) is $\d'$-hyperbolic,
with $\d'$ a universal constant which just depends on
$\d$, $a$, $b$ and $c$.
\end{theorem}

\begin{proposition}\cite[Proposition 5.6]{MR} \label{Prop: qi-pole}
Suppose $X,Y$ are proper hyperbolic geodesic metric spaces and $f:X \to Y$ is a quasi-isometry.
If $X$ has a pole in $v$, then $Y$ has a pole in $f(v)$.
\end{proposition}

The following result follows from \cite[Theorem 5.2.15]{BS} and \cite[Proposition 5.10]{MR}.

\begin{theorem} \label{Th: pq-symmetric} Let $f : X \to Y$ be a quasi-isometric map of proper hyperbolic
geodesic metric spaces. Then $\partial_\infty X$ is uniformly perfect if and only if $\partial_\infty Y$ is uniformly perfect
(with respect to any visual metrics).
\end{theorem}


\begin{theorem} \label{th:sufic}
Let $X$ be a $(2\d_2)$-regular non-exceptional Riemann surface,
$0 < \varepsilon < \Arcsinh 1$ and $0<\d<\min\{\d_1,\d_2\}$.
Assume that $[X]/\sim_\d$ is hyperbolic and has a pole.
Then $X$ has LII if and only if $\partial_\infty ([X]/\sim_\d)$ is uniformly perfect.
\end{theorem}

\begin{proof}
Since $[X]/\sim_\d$ is hyperbolic and has a pole, Theorems \ref{Th:qi-graph} and \ref{th: stability_hyp}, and Proposition \ref{Prop: qi-pole}
give that $\G_\d(X)$ is hyperbolic and has a pole.

Theorems \ref{Th:qi-graph} and \ref{Th: pq-symmetric} give that
$\partial_\infty ([X]/\sim_\d)$ is uniformly perfect if and only if
$\partial_\infty \G_\d(X)$ is uniformly perfect.
By Theorem \ref{t:iigraph}, this holds if and only if $\G_\d(X)$ has LII.
Finally, by Theorem \ref{th:graph}, $\G_\d(X)$ has LII if and only if $X$ has LII.
\end{proof}

To determine if a geodesic metric space has a pole is not a difficult task.
There are many results that allow to determine if a non-exceptional Riemann surface is hyperbolic (see, e.g., \cite{PRT1}, \cite{PRT2}, \cite{PT}, \cite{RT1}, \cite{RT3}, \cite{T}).
Thus, Theorem \ref{th:carcthyp} below is useful in order to apply Theorem \ref{th:sufic},
since it characterizes the hyperbolicity of $[X]/\sim_\d$ in terms of the hyperbolicity of $X$.

In order to prove Theorem \ref{th:carcthyp} we need some technical results.
The following three propositions appear in
\cite[Theorems 2.1 and 2.4]{RT1} and \cite[Lemma 3.1]{PRT2}.

\begin{proposition} \label{t:rt1}
Let $X$ be a geodesic metric space, $\{K_n\}_n$ compact subsets of $X$ and
$X'$ the quotient space obtained from $X$ by identifying the points of each
$K_n$ in a single point $k_n$.
Assume that there are positive constants $c_1,\,c_2,$ such that
$\diam_X K_n\le c_1$ and $d_X (K_n,K_m)\ge c_2$ if $n\neq m$,
and that $X'$ is a geodesic metric space.
Then $X$ and $X'$ are quasi-isometric, and $X$ is hyperbolic if and only if $X'$ is hyperbolic.
Furthermore, if $X$ (respectively, $X'$) is $\d$-hyperbolic, then $X'$ (respectively, $X$) is $\d'$-hyperbolic,
with $\d'$ a universal constant which just depends on
$\d$, $c_1$ and $c_2$.
\end{proposition}

Let us consider a geodesic metric space $X$, a family of geodesic metric subspaces
$\{X_n\}_n \subset X$ such that $\cup_n X_n=X$, $\eta_{nm}:=\eta_{mn}:=X_n\cap X_m$
are compact sets, and positive constants $c_1,\,c_2$.
We say that $\{X_n\}_n$ is a ($c_1,c_2$)-\emph{decomposition} of $X$
if $X \setminus \eta_{nm}$ is not connected for each non-empty set $\eta_{nm}$,
$\diam_{X_n} (\eta_{nm})\le c_1$ for every $n,m,$
and $d_{X_n} (\eta_{nm},\eta_{nk})\ge c_2$ for every $n$ and $m\neq k$ with non-empty sets $\eta_{nm}$ and $\eta_{nk}$.

\begin{proposition} \label{t:rt1bis}
Let $X$ be a geodesic metric space and $\{X_n\}_n \subset X$ a family of geodesic metric spaces which is a $(c_1,c_2)$-decomposition of $X$.
Then $X$ is hyperbolic if and only if there exists a constant $c_3$ such that $X_n$ is $c_3$-hyperbolic for every $n$.
Furthermore, if there exists such a constant $c_3$, then $X$ is $\d$-hyperbolic with $\d$ a universal constant which only depends on $c_1$, $c_2$ and $c_3$;
if $X$ is $\d$-hyperbolic, then $c_3$ is a universal constant which only depends on $c_1$, $c_2$ and $\d$.
\end{proposition}

Let $X$ be a non-exceptional Riemann surface.
For each choice of doubly connected domains $\{V_j\}_j$ in $X$ we define
$$
\aligned
D_{X}(\{V_j\}_j):=\sup_{j}\big\{
d_{V_j}(\eta^j_1,\eta_2^j):
\; & \eta_1^j,\eta_2^j
\text{ are the connected components of $\p V_j$} \\
& \; \text{and }
\; X\setminus\eta^j_1
\text{ is connected} \,
 \big\} .
\endaligned
$$

\begin{proposition} \label{l:prt2}
Let $X$ be a non-exceptional Riemann surface, $\{V_j\}_j$ doubly connected domains in $X$,
and $r,s$ positive constants with
$L(\p V_j) \le r$ for every $j$, and $d(V_j, V_k) \ge s$ for every $j \neq k$.
If $D_{X}(\{V_j\}_j)=\infty$, then $X$ is not hyperbolic.
\end{proposition}

\begin{lemma} \label{l:hypcollars}
Let $X$ be a non-exceptional Riemann surface and $0 < \d < \e< \Arcsinh 1$.
Then there exists a universal constant $\D$ such that $C_i$ and $K_j$ are $\D$-hyperbolic for every $i \in I$ and $j \in J_\d$.
\end{lemma}

\begin{proof}
Given $i \in I$, let us consider a geodesic ray $g_i$ in $C_i$ joining $\p C_i$ with the cusp $r_i$, and the inclusion $h: g_i \rightarrow C_i$.
Given $j \in J_\d$, let us consider a geodesic $\g_j$ in $K_j$ joining the two connected components of $\p K_j$, and the inclusion $h: \g_j \rightarrow K_j$.
If $\eta$ is a connected component of $\p C_i$ or $\p K_j$, then Lemma \ref{l:dc} gives that $L(\eta) < 2$.
Thus, in both cases $h$ is a $1$-full $(1,0)$-quasi-isometry and, since $g_i$ and $\g_j$ are $0$-hyperbolic,
Theorem \ref{th: stability_hyp} gives the result.
\end{proof}

For a non-exceptional Riemann surface $X$, let us define
$$
J_\d^c
= \big\{ j\in J_\d : \, X \setminus \g_j \text{ is connected} \, \big\} \, ,
\qquad
\L(X)
=\inf \big\{ L(\g_j) : \, j \in J_\d^c \, \big\} \, .
$$

\begin{theorem} \label{th:carcthyp}
Let $X$ be a non-exceptional Riemann surface,
$0 < \varepsilon < \Arcsinh 1$ and $0<\d< \d_1$.
Then $X$ is hyperbolic if and only if $[X]/\sim_\d$ is hyperbolic and
$\L(X) > 0$.
\end{theorem}

\begin{proof}
Lemma \ref{l:dc} gives that if $\eta$ is a connected component of $\big(\cup_{i\in I}\p C_i\big) \cup \big(\cup_{j \in J_\d} \p K_j\big)$, then $L(\eta) < 2$,
and so, $\diam \eta \le \frac12 L(\eta) < 1$.
For each collar $K_j=C(\g_j,h_j)$ of a simple closed geodesic $\g_j$ with $j \in J_\d$, consider $\partial K_j=\eta^j_1\cup \eta^j_2$ (where $\eta^j_i$ are Jordan curves).
We have, for $j \in J_\d$,
$$
d_{K_j}(\eta^j_1, \eta^j_2)
= 2h_j
=2\Arccosh \frac{\sinh \e}{\sinh (L(\g_j)/2)}
> 2\Arccosh \frac{\sinh \e}{\sinh \d} >0,
$$
since $0< \d < \d_1 < \e$.

Assume first that $\L(X) = 0$.
Since $\inf_{j \in J_\d^c} L(\g_j) = 0$, we have
$$
D_X \big(\{K_j\}_{j \in J_\d^c} \big)
= \sup_{j \in J_\d^c} d_{K_j}(\eta^j_1, \eta^j_2)
= \sup_{j \in J_\d^c} 2h_j
= \sup_{j \in J_\d^c} 2\Arccosh \frac{\sinh \e}{\sinh (L(\g_j)/2)}
= \infty ,
$$
and we conclude, by Proposition \ref{l:prt2}, that $X$ is not hyperbolic.

Consequently, it suffices to prove that if $\L(X) > 0$, then $X$ is hyperbolic if and only if $[X]/\sim_\d$ is hyperbolic.
Denote by $\{X_r\}_r$ the connected components of $X \setminus \big( \big(\cup_{i\in I}\p C_i\big) \cup \big(\cup_{j \in J_\d \setminus J_\d^c} \p K_j\big) \big)$.

Lemma \ref{l:us} gives that if $\s_1,\s_2 \subset \p X_r$ are connected components of $\big(\cup_{i\in I}\p C_i\big) \cup \big(\cup_{j \in J_\d \setminus J_\d^c} \p K_j\big)$
and they belong to the closure of different collars, then
$d_{X_r}(\s_1, \s_2) \ge d_{X}(\s_1, \s_2) \ge 2\log \frac{1}{\sinh \e}$.

Therefore, if we define $c_2=\min\{2\Arccosh \frac{\sinh \e}{\sinh \d}, \,2\log \frac{1}{\sinh \e}\}$, we have that
$\{ X_r\}_{r}$, $\{ C_i\}_{i\in I}$, $\{ K_j\}_{j \in J_\d \setminus J_\d^c}$
is a ($1,c_2$)-decomposition of $X$.
By Proposition \ref{t:rt1bis} and Lemma \ref{l:hypcollars}, $X$ is hyperbolic if and only if
there exists a constant $c_3$ such that $X_r$ is $c_3$-hyperbolic for every $r$.

For each $r$, let $X_r'$ be the geodesic metric space obtained from $X_r$
by identifying the points in the collar $K_j$ in a single point $x_j$ for every $j \in J_\d^c$ with $K_j \subset X_r$.
Let us define the (non-smooth) surface $Y_r$
obtained from $X_r$ by removing the collar $K_j$ and by gluing $\eta^j_1$ and $\eta^j_2$ for every $j \in J_\d^c$ with $K_j \subset X_r$;
consider on $Y_r$ the inner metric inherited from $X_r$.
Let $Y_r'$ be the geodesic metric space obtained from $Y_r$
by identifying the points in $\p K_j$ in a single point $x_j$ for every $j \in J_\d^c$ with $K_j \subset X_r$.

Since $\L(X) > 0$, we have
$$
D_X \big(\{K_j\}_{j \in J_\d^c} \big)
= \sup_{j \in J_\d^c} 2h_j
= \sup_{j \in J_\d^c} 2\Arccosh \frac{\sinh \e}{\sinh (L(\g_j)/2)}
= 2\Arccosh \frac{\sinh \e}{\sinh (\L(X)/2)}
< \infty
$$
and so,
$$
\sup_{j \in J_\d^c} \diam K_j
\le \frac12 \,L(\eta_1^j) + \sup_{j \in J_\d^c} 2h_j + \frac12 \,L(\eta_2^j)
\le 2\Arccosh \frac{\sinh \e}{\sinh (\L(X)/2)} +2
< \infty.
$$
This fact and Proposition \ref{t:rt1} give that
$X_r$ is hyperbolic if and only if $X_r'$ is hyperbolic, and
$Y_r$ is hyperbolic if and only if $Y_r'$ is hyperbolic, with uniform hyperbolicity constants.
Since $X_r'=Y_r'$,
we conclude that
there exists a constant $c_3$ such that $X_r$ is $c_3$-hyperbolic for every $r$ if and only if
there exists a constant $c_3'$ such that $Y_r$ is $c_3'$-hyperbolic for every $r$.

Since $\{ Y_r\}_{r}$
is a ($1,c_2$)-decomposition of $[X]/\sim_\d$,
Proposition \ref{t:rt1bis} gives that $[X]/\sim_\d$ is hyperbolic if and only if
there exists a constant $c_3'$ such that $Y_r$ is $c_3'$-hyperbolic for every $r$.

Hence, $X$ is hyperbolic if and only if $[X]/\sim_\d$, and the conclusion holds.
\end{proof}

Our last result shows that the hypotheses on Theorem \ref{th:sufic} do not depend on $\d$.

\begin{proposition} \label{p:carcthyp}
Let $X$ be a non-exceptional Riemann surface,
$0 < \varepsilon < \Arcsinh 1$ and $0< \d, \hat\d < \d_1$.
The following statements hold:

$(1)$ $[X]/\sim_{\hat\d}$ is hyperbolic if and only if $[X]/\sim_\d$ is hyperbolic.

$(2)$ $[X]/\sim_{\hat\d}$ has a pole if and only if $[X]/\sim_\d$ has a pole.
\end{proposition}

\begin{proof}
By Theorem \ref{th: stability_hyp} and Proposition \ref{Prop: qi-pole}, it suffices to prove that
$[X]/\sim_{\hat\d}$ and $[X]/\sim_\d$ are quasi-isometric.

Without loss of generality we can assume that $\d < \hat\d$.
Thus, $J_\d \subseteq J_{\hat\d}$.

Let $([X]/\sim_\d)'$ be the geodesic metric space obtained from $[X]/\sim_\d$
by identifying the points in the closure of the collar $K_j$ in a single point $x_j$ for every $j \in J_{\hat\d} \setminus J_\d$, and
$([X]/\sim_{\hat\d})'$ the geodesic metric space obtained from $[X]/\sim_{\hat\d}$
by identifying the points in $\eta_1^j \cong \eta_2^j$ in a single point $x_j$ for every $j \in J_{\hat\d} \setminus J_\d$.
Thus, $([X]/\sim_{\hat\d})' = ([X]/\sim_\d)'$.

Lemma \ref{l:us} gives $d_{[X]/\sim_\d}(K_{j_1}, K_{j_2}) \ge 2\log \frac1{\sinh \e}$ for every $j_1,j_2 \in J_{\hat\d} \setminus J_\d$.
Also, Lemmas \ref{l:0} and \ref{l:dc} give
$$
\begin{aligned}
\diam_{[X]/\sim_\d} K_j
& \le \frac12 \,L(\eta_1^j) + 2h_j + \frac12 \,L(\eta_2^j)
\le 2\Arccosh \frac{\sinh \e}{\sinh (L(\g_j)/2)} + 2 ,
\\
& \le 2\Arccosh \frac{\sinh \e}{\sinh \d} + 2 ,
\end{aligned}
$$
for every $j \in J_{\hat\d} \setminus J_\d$.

These inequalities and Proposition \ref{t:rt1} give that
$[X]/\sim_\d$ and $([X]/\sim_\d)'$ are quasi-isometric, and
$[X]/\sim_{\hat\d}$ and $([X]/\sim_{\hat\d})'$ are quasi-isometric.
Since $([X]/\sim_{\hat\d})' = ([X]/\sim_\d)'$,
we conclude that
$[X]/\sim_{\hat\d}$ and $[X]/\sim_\d$  are quasi-isometric.
\end{proof}


\begin{thebibliography}{99}


%
%

\bibitem{AR} Alvarez, V., Rodríguez, J. M., Structure theorems for Riemann and topological surfaces,
{\it J. London Math. Soc.} {\bf 69} (2004), 153-168.

\bibitem{A1} Ancona, A., Negatively curved manifolds, elliptic operators, and Martin boundary,
{\it Annals of Math.} {\bf 125} (1987) 495-536.

\bibitem{A2} Ancona, A., Positive harmonic functions and hyperbolicity, in Potential Theory, Surveys
and Problems, eds. J. Kr\'al et al., Lecture Notes in Math., No. 1344, Springer-Verlag
(1988), pp. 1-24.

\bibitem{A3} Ancona, A., Theorie du potentiel sur les graphes et les varieties, in Ecol\'e d'Et\'e de
Probabilit\'es de Saint-Flour XVII-1988, eds. A.Ancona et al., Lecture Notes in Math.,
No. 1427, Springer-Verlag (1990).

%

\bibitem{BGS} Ballman, W., Gromov, M., Schroeder, V., Manifolds of Non-positive Curvature, Birkhauser, Boston, 1985.

\bibitem{Be} Bers, L., An Inequality for Riemann Surfaces. Differential Geometry and Complex Analysis. H. E. Rauch Memorial Volume. Springer-Verlag, 1985.

%

\bibitem{BJ} Bishop, C. J., Jones, P. W.,  Hausdorff dimension and Kleinian groups,
{\it Acta Math.} {\bf 179} (1997), 1-39.

\bibitem{BH} Bridson, M., Haefliger, A., Metric spaces of non-positive curvature.
Springer-Verlag, Berlin, 1999.

%
%
%

\bibitem{Bu1} Buser, P., A note on the isoperimetric constant,
{\it Ann. Sci. \'Ecole Normale Sup.} {\bf 15} (1982), 213-230.

\bibitem{Bu} Buser, P., Geometry and Spectra of Compact Riemann Surfaces. Birkh\"auser, Boston, 1992.

\bibitem{BS}  Buyalo, S.,  Schroeder, V., \emph{Elements of Asymptotic
Geometry.} EMS Monographs in Mathematics. Germany, 2007.

\bibitem{CGPR} Cant\'on, A., Granados, A., Portilla, A., Rodr\'{\i}guez, J. M.,
Quasi-isometries and isoperimetric inequalities in planar domains,
{\it J. Math. Soc. Japan} {\bf 67} (2015), 127-157.

%

\bibitem{C1} Chavel, I., Eigenvalues in Riemannian Geometry. Academic Press, New York, 1984.

\bibitem{C2} Chavel, I.,
Isoperimetric inequalities: differential geometric and analytic perspectives.
Cambridge University Press, Cambridge, 2001.

\bibitem{Ch} Cheeger, J., A lower bound for the smallest eigenvalue of
the Laplacian. In {\it Problems in Analysis}, Princeton University
Press, Princeton (1970), 195-199.

%
%
%
%

\bibitem{DMT} Dress, A., Moulton, V., Terhalle, W., T-theory: an overview, {\it Europ. J. Combin.} {\bf 17} (1996), 161-175.

\bibitem{FM1} Fern\'andez, J. L., Meli\'an, M. V.,
Bounded geodesics of Riemann surfaces and hyperbolic manifolds,
{\it Trans. Amer. Math. Soc.} {\bf 347} (1995), 3533-3549.

\bibitem{FM2} Fern\'andez, J. L., Meli\'an, M. V.,
Escaping geodesics of Riemannian surfaces,
{\it Acta Math.} {\bf 187} (2001), 213-236.

\bibitem{FMP1} Fern\'andez, J. L., Meli\'an, M. V., Pestana, D.,
Quantitative mixing results and inner functions, {\it Math. Ann.} {\bf 337} (2007), 233-251.

\bibitem{FMP2} Fern\'andez, J. L., Meli\'an, M. V., Pestana, D.,
Expanding maps, shrinking targets and hitting times,
{\it Nonlinearity} {\bf 25} (2012), 2443-2471.

\bibitem{FR1}
Fern\'andez, J. L., Rodr{\'\i}guez, J. M.,
The exponent of convergence of Riemann surfaces: Bass Riemann surfaces,
{\it Annales Acad. Sci. Fenn. A. I.} {\bf 15} (1990), 165-183.

\bibitem{GH} Ghys, E., de la Harpe, P.,
{\it Sur les Groupes Hyperboliques d'apr\`es Mikhael Gromov.}
Progress in Mathematics, Volume 83, Birkh\"auser, 1990.

\bibitem{GPPRT} Granados, A., Pestana, D., Portilla, A., Rodr\'{\i}guez, J. M., Tour\'{\i}s, E.,
Stability of the injectivity radius under quasi-isometries and applications to isoperimetric inequalities,
{\it RACSAM}, in press, DOI 10.1007/s13398-017-0417-4

\bibitem{GPPR} Granados, A., Pestana, D., Portilla, A., Rodr\'{\i}guez, J. M., Stability of $p$-parabolicity under quasi-isometries.
Submitted.

%
%
%
%

\bibitem{K22} Jonckheere, E. A. and Lohsoonthorn, P., Geometry of network security,
{\it Amer. Control Conf.} {\bf ACC} (2004), 111-151.

\bibitem{K} Kanai, M., Rough isometries and combinatorial approximations of geometries of noncompact
Riemannian manifolds, {\it J. Math. Soc. Japan} {\bf 37} (1985), 391-413.

%
%
%
%
%


\bibitem{MR} Mart\'inez-P\'erez, A. Rodr\'iguez J. M., Cheeger isoperimetric constant of Gromov hyperbolic manifolds and graphs. Commun. Contemp. Math. https://doi.org/10.1142/S021919971750050X


\bibitem{Matsuzaki} Matsuzaki, K., Isoperimetric constants for conservative Fuchsian groups,
{\it Kodai Math. J.} {\bf 28} (2005), 292-300.

\bibitem{MRT} Meli\'an, M. V., Rodr{\'\i}guez, J. M., Tour{\'\i}s, E.,
Escaping geodesics in Riemannian surfaces with pinched variable negative curvature.
Submitted.

\bibitem{MoSoVi} Montgolfier, F., Soto, M. and Viennot, L., Treewidth and Hyperbolicity of the Internet, In: 10th IEEE
International Symposium on Network Computing and Applications (NCA), 2011, pp. 25–32.

\bibitem{P} Paulin, F., On the critical exponent of a discrete group of hyperbolic isometries,
\emph{Differ. Geom. Appl.} {\bf 7} (1997), 231-236.

\bibitem{Po} P\'olya, G.,
Isoperimetric inequalities in mathematical physics.
Princeton University Press, 1951.

\bibitem{PRT1} Portilla, A., Rodr\'{\i}guez, J. M., Tour{\'\i}s, E.,
Gromov hyperbolicity through decomposition of metric spaces II,
{\it J. Geom. Anal.}  {\bf 14} (2004), 123-149.

\bibitem{PRT2} Portilla, A., Rodr\'{\i}guez, J. M., Tour{\'\i}s, E.,
The topology of balls and Gromov hyperbolicity of Riemann surfaces,
{\it Diff. Geom. Appl.} {\bf 21} (2004), 317-335.

\bibitem{PT} Portilla, A., Tour{\'\i}s, E.,
A characterization of Gromov hyperbolicity of surfaces with variable negative curvature,
{\it Publ. Mat.} {\bf 53} (2009), 83-110.

\bibitem{R} Randol, B., Cylinders in Riemann surfaces, {\it Comment. Math. Helv.} {\bf 54} (1979), 1-5.


\bibitem{RT1} Rodr\'{\i}guez, J. M., Tour{\'\i}s, E.,
Gromov hyperbolicity through decomposition of metric spaces,
{\it Acta Math. Hung.}~{\bf 103} (2004), 53-84.

\bibitem{RT3} Rodr\'{\i}guez, J. M., Tourís, E., Gromov hyperbolicity of Riemann surfaces,
{\it Acta Math. Sinica} {\bf 23} (2007), 209-228.

%
%

\bibitem{Sh} Shimizu, H., On discontinuous groups operating on the product of upper half-planes, {\it Ann. of Math.} {\bf 77} (1963), 33-71.

\bibitem{S} Sullivan, D., Related aspects of positivity in Riemannian geometry, {\it J. Diff. Geom.} {\bf 25} (1987), 327-351.

\bibitem{T} Tour{\'\i}s, E., Graphs and Gromov hyperbolicity of non-constant negatively curved surfaces,
{\it J. Math. Anal. Appl.} {\bf 380} (2011), 865-881.


\bibitem{WZ} Wu, Y. and Zhang, C.,
Chordality and hyperbolicity of a graph, {\it Electr. J. Comb.} {\bf 18} (2011), P43.

\end{thebibliography}
\end{document}